\documentclass[11pt]{amsart}
\usepackage{amssymb}
\usepackage[margin=1in]{geometry}
\usepackage[colorlinks]{hyperref}
\newtheorem{theorem}{Theorem}[section]

\newtheorem{prop}[theorem]{Proposition}
\newtheorem{cor}[theorem]{Corollary}

\theoremstyle{definition}
\newtheorem{definition}[theorem]{Definition}

\theoremstyle{remark}
\newtheorem{remark}[theorem]{Remark}

\numberwithin{equation}{section}
\begin{document}

\newcommand{\spacing}[1]{\renewcommand{\baselinestretch}{#1}\large\normalsize}
\spacing{1.14}

\title{The Relation Between Automorphism Group
and Isometry Group of Left Invariant $ (\alpha,\beta)$-metrics}

\author {M. Nejadahmad}

\address{Masumeh Nejadahmad, Department of Mathematics, Isfahan University of Technology, Iran} \email{masumeh.nejadahmad@math.iut.ac.ir}

\author {H. R. Salimi Moghaddam}

\address{Hamid Reza Salimi Moghaddam, Department of Mathematics, Faculty of  Sciences, University of Isfahan,\\ Isfahan, 81746-73441, Iran. } \email{hr.salimi@sci.ui.ac.ir and salimi.moghaddam@gmail.com}

\keywords{Automorphism group, Isometry group, Lie group, Left Invariant $ (\alpha,\beta)$-metric.\\
AMS 2010 Mathematics Subject Classification: 53C60, 58B20, 20F28}

\date{\today}
\begin{abstract}
This work generalizes the results of an earlier paper by the second author, from Randers metrics to $(\alpha,\beta)$-metrics. Let $F$ be an $(\alpha,\beta)$-metric which is defined by a left invariant vector field and a left invariant Riemannian metric on a simply connected real Lie group $G$. We consider the automorphism and isometry groups of the Finsler manifold $(G,F)$ and their intersection. We prove that for an arbitrary left invariant vector field $X$ and any compact subgroup $K$ of automorphisms which $X$ is invariant under them, there exists an $(\alpha,\beta)$-metric such that $K$ is a subgroup of its isometry group.
\end{abstract}

\maketitle
%%---------------------------INTRODUCTION--------------------------

\section{\textbf{Introduction}}

$(\alpha,\beta)$-metrics are one of the most attractive Finsler metrics because they have many applications in physics (see \cite{Asanov, Ingarden, Landau-Lifshitz}) and also they have simple structure. In fact they are expressed in terms of a Riemannian metric $\alpha =\sqrt{g_{ij}y^{i}y^{j}} $ and a 1-form $ \beta $ on a Riemannian manifold $(M,g)$ as follows:
\begin{equation}
    F = \alpha \phi(\frac{\beta}{\alpha}),
\end{equation}
where $\phi$ is a positive smooth map on $(-b_{0} , b_{0})\subset\Bbb{R}$ (see \cite{Chern-Shen}). $ F=\alpha \phi (\frac{\beta}{\alpha})$  is a Finsler metric if and only if $ \| \beta \|_{\alpha} < b_{0} $ and
\begin{equation}\label{inequality}
    \phi(s) - s\phi'(s) +(b^{2}-s^{2}) \phi''(s)>0, \ \ \mid s\mid \leq b < b_{0}.
\end{equation}
Some of famous Finsler metrics such as Randers, Kropina and Matsumoto metrics belong to this family. In definition of $(\alpha,\beta)$-metrics if we put $ \phi(s)=1+s$, $\phi(s)=\frac{1}{s}$ or $\phi(s)= \frac{1}{1-s}$, then we obtain Randers metric $F=\alpha+\beta$, Kropina metric $F=\frac{\alpha^{2}}{\beta}$ and Matsumoto metric $F=\frac{\alpha^{2}}{\alpha-\beta}$, respectively, although Kropina metric is not a regular Finsler metric (for more details see \cite{Antonelli-Ingarden-Matsumoto, Bao-Chern-Shen} and \cite{Chern-Shen}).
Among Riemannian metrics, the left invariant metrics on Lie groups $G$ are of particular importance and have therefore attracted the attention of many mathematicians (see \cite{Eberlein, Kobayashi-Nomizu, Milnor} and \cite{Nomizu}). \\
During the two last decades some studies have been done on the left invariant Finsler metrics on Lie groups. Specially, many geometric properties of left invariant $(\alpha,\beta)$ metrics have been studied (\cite{Deng-book, Deng-Hosseini-Liu-Salimi, Deng-Hou, Deng-Hu, Esrafilian-Salimi Moghaddam} and \cite{Salimi}).\\
It's well known that for an arbitrary 1-form $\beta$ on a Riemannian manifold $(M,g)$ there exists a unique vector field $X$ on $M$ such that for all $x \in M$ and $y \in T_{x}M$ we have
\begin{equation}
    g(y, X(x)) = \beta(x,y).
\end{equation}
Using the vector fields instead of the 1-forms is very useful to define left invariant $(\alpha,\beta)$-metrics on Lie groups. If $(G,g)$ is a Lie group with a left invariant Riemannian metric and $X \in \mathfrak{g}$ is a left invariant vector field on $G$ such that $\parallel X \parallel_{g}< b_{0}$, then easily we can see the following $(\alpha,\beta)$-metric is left invariant (see \cite{Deng-book} and \cite{Deng-Hosseini-Liu-Salimi}):
\begin{equation}\label{left invariant metric}
    F(h, Y_{h}) =\sqrt{\langle Y_{h},Y_{h} \rangle} \phi(\frac{\langle X_{h},Y_{h} \rangle}{\sqrt{\langle Y_{h},Y_{h} \rangle}}),
\end{equation}
for all $h\in G$ and $Y\in \mathfrak{g}$, where for simplicity we have used the notation ${\langle \cdot,\cdot \rangle}$ for the left invariant Riemannian metric $g$.\\
In \cite{Salimi}, the second author has studied the intersection between automorphism and isometry groups in the case of left invariant Randers metrics. The purpose of this paper is to develop some results of the article \cite{Salimi} to the case of left invariant $(\alpha,\beta)$-metrics. In this work, in definition of $(\alpha,\beta)$-metric, we will consider the map $\phi$ is a one-to-one map. We can see Randers, Kropina and Matsumoto metrics satisfy this condition. \section{\textbf{Automorphism and Isometry Groups of $ (\alpha,\beta)$- Lie groups}}
For any simply connected real Lie group $G$ with Lie algebra
$\frak{g}$ the following homomorphism:
\begin{equation}\label{T}
    \left\{%
\begin{array}{ll}
    T:Aut(G)\longrightarrow Aut(\frak{g}), \\
    \psi\longrightarrow T(\psi)=(d\psi)_e:T_eG=\frak{g}\longrightarrow\frak{g}.\\
\end{array}%
\right.
\end{equation}
is an isomorphism of Lie groups (see \cite{Eberlein}).
So the isomorphism $T$ makes a correspondence between their maximal compact subgroups.\\
Now let $\psi$ be an automorphism of the Lie group $G$. Then for an arbitrary $g\in G$ we can see:
\begin{equation}\label{commute with left translation}
    \psi \circ L_{g}= L_{\psi (g)} \circ \psi.
\end{equation}
Here we recall the following definition given by the second author in \cite{Salimi}.
\begin{definition}
(see \cite{Salimi}, definition 3.1) Suppose that $\psi\in Aut(G) $ and $X$ is a vector field on the Lie group $G$. $X$ is called a $\psi$-invariant vector field if $d\psi \circ X=X \circ \psi$.
\end{definition}
From now on we assume that $F$ is an $(\alpha,\beta)$-metric defined by the relation (\ref{left invariant metric}) such that $\phi$ is a one-to-one map. Also $I(G,\langle\cdot ,\cdot \rangle)$ and $I(G,F)$ denote the isometry groups of the Riemannian manifold $(G,\langle\cdot , \cdot\rangle) $ and the Finsler manifold $(G,F)$, respectively.
\begin{prop}\label{M1}
If $ \psi \in I(G,\langle\cdot , \cdot\rangle)$, then $\psi \in I(G,F)$ if and only if $X$ is $\psi$-invariant.
\end{prop}
\begin{proof}
If $X$ is a $\psi$-invariant vector field then for any $g\in G$ we have:
\begin{eqnarray*}
F(\psi(g), d\psi_{g}Y_{g}) &=&  \sqrt{\langle d\psi_{g}Y_{g},d\psi_{g}Y_{g} \rangle}\phi(\frac{\langle X_{\psi(g)},d\psi_{g}Y_{g}\rangle}{\sqrt{\langle d\psi_{g}Y_{g},d\psi_{g}Y_{g} \rangle}}) \\
&=&\sqrt{\langle Y_{g},Y_{g} \rangle}\phi(\frac{\langle X_{\psi(g)},d\psi_{g}Y_{g}\rangle}{\sqrt{\langle Y_{g},Y_{g} \rangle}}) \\
&=&\sqrt{\langle Y_{g},Y_{g} \rangle}\phi(\frac{\langle d\psi_{g}X_{g},d\psi_{g}Y_{g}\rangle}{\sqrt{\langle Y_{g},Y_{g} \rangle}}) \\
&=&\sqrt{\langle Y_{g},Y_{g} \rangle} \phi(\frac{\langle X_{g},Y_{g} \rangle}{\sqrt{\langle Y_{g},Y_{g} \rangle}})\\
&=& F(g,Y_{g}).
\end{eqnarray*}
For the converse, consider $\psi\in I(G,F)$, so we have:
\begin{eqnarray*}
\sqrt{\langle Y_{g},Y_{g} \rangle} \phi(\frac{\langle X_{g},Y_{g} \rangle}{\sqrt{\langle Y_{g},Y_{g} \rangle}})&=& F(g,Y_{g}) \\
&=& F(\psi(g), d\psi_{g}Y_{g}) \\
&=&\sqrt{\langle d\psi_{g}Y_{g},d\psi_{g}Y_{g} \rangle}\phi(\frac{\langle X_{\psi(g)},d\psi_{g}Y_{g}\rangle}{\sqrt{\langle d\psi_{g}Y_{g},d\psi_{g}Y_{g}\rangle}}).
\end{eqnarray*}
On noting that $\phi$ is one-to-one and $ \psi \in I(G,\langle\cdot , \cdot\rangle) $, we have:
\begin{equation*}
    \langle d_{\psi(g)}X_{g},d\psi_{g}Y_{g} \rangle = \langle X_{g},Y_{g} \rangle = \langle X_{\psi(g)},d\psi_{g}Y_{g}\rangle,
\end{equation*}
therefore
\begin{equation*}
    \langle d_{\phi(g)}X_{g} - X_{\phi(g)}, d\psi_{g}Y_{g}\rangle =0.
\end{equation*}
Thus $ d_{\phi(g)}X_{g} = X_{\phi(g)} $ which shows that X is $ \psi $-invariant.
\end{proof}
\begin{remark}
In the above proposition, without the assumption $\phi$ is one-to-one, if the vector field $X$ is $\psi$-invariant then $\psi\in I(G,F)$.
\end{remark}
\begin{prop}
Suppose that $\langle\cdot , \cdot\rangle$ is an arbitrary left invariant Riemannian metric on a Lie group $G$ and $F=\alpha\phi(\frac{\beta}{\alpha})$ is the left invariant $(\alpha,\beta)$-metric defined by it and a vector field $X$, using the formula (\ref{left invariant metric}). If $\phi$ is a one-to-one map then the vector field $X$ is a left invariant vector field.
\end{prop}
\begin{proof}
Let $g,h \in G$, then we have:
\begin{eqnarray*}
F(h,Y_{h})&=&\sqrt{\langle Y_{h},Y_{h} \rangle} \phi(\frac{\langle X_{h},Y_{h} \rangle}{\sqrt{\langle Y_{h},Y_{h} \rangle}})\\
&=&\sqrt{\langle dL_{g}Y_{h},dL_{g}Y_{h} \rangle}\phi(\frac{\langle X_{L_{g}h},dL_{g}Y_{h} \rangle}{\sqrt{\langle dL_{g}Y_{h},dL_{g}Y_{h} \rangle}})\\
&=& F(L_{g}h, dL_{g}Y_{h}).
\end{eqnarray*}
$\phi$ is one-to-one so we have:
\begin{equation*}
    \langle X_{h},Y_{h} \rangle = \langle X_{L_{g}h},dL_{g}Y_{h} \rangle,
\end{equation*}
which means that
\begin{equation*}
    \langle X_{h} - dL_{g^{-1}}X_{L_{g}h},Y_{h}\rangle =0.
\end{equation*}
The last equation shows that $X$ is a left invariant vector field.
\end{proof}
Suppose that $X$ is a left invariant vector field on a simply connected Lie group $G$ with the unit element $e$. Now we recall the following notations from \cite{Salimi}:
\begin{eqnarray*}
  Aut_{X}(G)&=&\lbrace  \psi \in Aut(G) | X ~is~ \psi-invariant\rbrace, \\
  Aut_{X}(\mathfrak{g})&=&\lbrace d\psi_{e} \in Aut(\mathfrak{g}) | d\psi_{e} X_{e}=X_{e}\rbrace.
\end{eqnarray*}

\begin{prop}
Let $G$ be a Lie group with Lie algebra $\mathfrak{g}$, $ \langle\cdot , \cdot\rangle $ be a left invariant Riemannian metric on it and $X$ be a left invariant vector field on $G$ such that $\sqrt{\langle X ,X \rangle}< b_{0}$. Suppose that $F$ is a left invariant $(\alpha,\beta)$-metric defined by $\langle\cdot , \cdot\rangle$, $X$ and a one-to-one map $\phi$. Assume that $\psi$ is an automorphism of $G$ such that $T(\psi)=d\psi_{e}\in Aut_{X}(\mathfrak{g})$ and also $d\psi_{e}$ is a linear isometry of the inner product space $(\mathfrak{g}, \langle\cdot , \cdot\rangle)$. Then, $ \psi\in I(G,F)$.
\end{prop}
\begin{proof}
Similar to the proof of proposition 3.8 of \cite{Salimi} we have $\psi\in I(G,\langle\cdot , \cdot\rangle)$.
%Using the equation (\ref{commute with left translation}), for any $ g \in G$ we have $\psi \circ L_{g}= L_{\psi(g)}\circ \psi$. So
%\begin{eqnarray*}
%d\psi_{g}X_{g} &=& d(\psi \circ L_{g})_{e}X_{e} \\
%&=& d(L_{\psi(g)} \circ \psi)_{e}X_{e} \\
%&=& (dL_{\psi(g)})_{e}(d\psi _{e}X_{e}) \\
%&=& (dL_{\psi(g)})_{e} X_{e} = X_{\psi(g)},
%\end{eqnarray*}
%which shows that $\psi\in  Aut_{X}(G)$. Moreover, we have
%\begin{eqnarray*}
%\langle (d\psi)_{g}Y_{g} , (d\psi) _{g}Z_{g}\rangle &=& \langle (d\psi)_{g} (dL_{g})_{e}Y_{e},(d\psi)_{g} (dL_{g})_{e}Z_{e}  \rangle  \\
%&=& \langle (dL_{\psi(g)})_{e}(d\psi)_{e}Y_{e}, (dL_{\psi(g)})_{e}(d\psi)_{e}Z_{e} \rangle \\
%&=& \langle Y_{g},Z_{g}\rangle,
%\end{eqnarray*}
%hence $\psi\in I(G,\langle\cdot , \cdot\rangle)$.
Now the proposition \ref{M1} proves $\psi\in I(G,F)$.
\end{proof}
From the above proposition we conclude that the isomorphism $ T: Aut(G) \rightarrow Aut(\mathfrak{g}) $ maps
$K=Aut_{X}(G)\cap I(G, F)$ onto $K'=Aut_{X}(\mathfrak{g}) \cap O(\mathfrak{g})$, where we use the symbol $O(\mathfrak{g})$ to represent the orthogonal group of $(\mathfrak{g},\langle\cdot , \cdot\rangle)$. Also we have the following corollary.
\begin{cor}
If $K' =Aut_{X}(\mathfrak{g})\cap O(\mathfrak{g})$ then both $K'$ and $K =T^{-1}(K')$ are compact Lie groups.
\end{cor}
\begin{proof}
Suppose that the function $\Upsilon: Aut(\mathfrak{g})\longrightarrow \mathfrak{g}$ is defined by $\Upsilon(\psi):=\psi(X_e)-X_e$. Easily we see that $\Upsilon$ is a continuous map. Therefor $Aut_{X}(\mathfrak{g})=\Upsilon^{-1}\{0\}$ is a closed subset of $Aut(\mathfrak{g})$. So $K'$ is a closed subgroup of $O(\mathfrak{g}) $.
\end{proof}
\begin{prop}
Suppose that $X$ is a left invariant vector field on $G$. If $K$ is an arbitrary compact subgroup of $Aut_{X}(G)$ then there exists a left invariant $(\alpha,\beta)$-metric $F$ on $G$ such that $K\subset Aut_{X}(G)\cap I(G,F)$, where $F$ is defined by the left invariant vector field $X$, a left invariant Riemannian metric $\langle\cdot , \cdot\rangle$, and a one-to-one map $\phi$ such that $\langle X , X\rangle<b_0$.
\end{prop}
\begin{proof}
In \cite{Eberlein}, it is shown that there exists a left invariant Riemannian metric $\langle\cdot ,\cdot\rangle_{0}$ on $G$ such that $K\subset Aut(G)\cap I(G,\langle\cdot,\cdot\rangle_{0})$. Thus $K \subset Aut_{X}(G)\cap I(G,\langle,\rangle_{0})$. Now we choose $N\in\Bbb{N}$ sufficiently large such that $\sqrt{\frac{1}{N}\langle X,X\rangle_{0} }< b_{0}$. Let $\langle\cdot , \cdot\rangle = \frac{1}{N}\langle\cdot , \cdot\rangle_{0}$. Clearly,  the Riemannian manifolds $(G,\langle\cdot , \cdot\rangle)$ and $(G,\langle\cdot , \cdot\rangle_{0})$ have the same isometry groups. If $F$ is the $(\alpha,\beta)$-metric defined by $X$, the Riemannian metric $\langle\cdot , \cdot\rangle$ and the map $\phi$, then proposition \ref{M1} completes the proof.
\end{proof}
\begin{prop}
Assume that $G$ is a simply connected Lie group and $X$ is a left invariant vector field on $G$. Then
there exists a left invariant $(\alpha,\beta)$-metric $F$ such that $K= Aut_{X}(G)\cap I(G, F) $ is a maximal compact subgroup of $Aut_{X}(G)$. Moreover $F$ is defined by the left invariant vector field $X$, a left invariant Riemannian metric $\langle\cdot , \cdot\rangle$, and a one-to-one map $\phi$ such that $\langle X , X\rangle<b_0$.
\end{prop}
\begin{proof}
Suppose that $K$ is a maximal compact subgroup of $Aut_{X}(G)$.  The above proposition shows that there is a left invariant $(\alpha,\beta)$-metric on $G$ such that $K\subset Aut_{X}(G)\cap I(G,F) $. Now the maximality of $K$ completes the proof.
\end{proof}

As example, we see that all the above results are valid for Randers, Kropina and Matsumoto metrics (Also see \cite{Salimi} for Randers metric).

\end{document}